\documentclass[reqno, 12pt]{amsart}
\usepackage{enumerate}
\usepackage{amssymb,url,color}
\usepackage{hyperref}
\usepackage{amsmath,amsthm}

\allowdisplaybreaks[4]

\newtheorem{thm}{Theorem}[section]

\newtheorem{lem}{Lemma}[section]
\newtheorem{rem}{Remark}[section]

\theoremstyle{definition}

\numberwithin{equation}{section}
\newcommand{\pp}{\mathbb{P}}
\newcommand{\nn}{\mathbb {N}}
\newcommand{\ee}{\mathbb{E}}

\newcommand{\FF}{\mathcal{F}}

\newcommand{\rr}{\mathbb{R}}

\def\beq{\begin{equation}}
\def\deq{\end{equation}}
\def\dsp{\displaystyle}
\allowdisplaybreaks 

\bfseries\rmfamily 

\begin{document}
\title[Moderate deviation principle for stochastic approximation ]
{Moderate Deviation Principle for a Stochastic Approximation Process}
\thanks{This work is supported by National Natural Science Foundation of China (NSFC-11971154) and Natural Science Foundation of Henan (No. 262300421307).}

\author[J. N. Shi]{Jianan Shi}
\address[J. N. Shi]{School of Mathematics and Statistics, Henan Normal University, Henan Province, 453007, China.}
\email{\href{mailto: J. N. Shi
<jiananshi2022@126.com>}{jiananshi2022@126.com}}

\author[Q. Yin]{Qing Yin}
\address[Q. Yin] {School of Mathematics and Statistics, Henan Normal University, Henan Province, 453007, China.}
\email{\href{mailto: Q. Yin
<qingyin1282@163.com>}{qingyin1282@163.com}}

\author[Y. Miao]{Yu Miao}
\address[Y. Miao]{School of Mathematics and Statistics, Henan Normal University, Henan Province, 453007, China.} \email{\href{mailto: Y. Miao
<yumiao728@gmail.com>}{yumiao728@gmail.com}; \href{mailto: Y. Miao <yumiao728@126.com>}{yumiao728@126.com}}

\begin{abstract}
In this paper, we investigate a stochastic approximation procedure $\left(X_n\right)_{n\ge 0}$ taking values in $\rr$. The process is adapted to a filtration $(\FF_n)_{n\ge 0}$ and satisfies the recursion $X_{n+1}=X_n+\frac{b}{n+1}\big[g(X_n)+U_{n+1}\big]$, where $b>0, g:\rr\to\rr$ is a function and $\left(U_n\right)_{n\ge 1}$ is a sequence of bounded martingale differences adapted to the filtration $(\FF_n)_{n\ge 1}$. We establish the moderate deviation principle for the stochastic process $(X_n)_{n\ge 0}$.  As auxiliary results, we also obtain the exponential inequality for $(X_n)_{n\ge 0}$ and the moderate deviation principle for weighted sums of  bounded martingale differences.
\end{abstract}

\keywords{Stochastic approximation; moderate deviation principle; bounded martingale difference.}
\subjclass[2020]{62L20, 60F10}

\maketitle
\section{Introduction}
We consider the stochastic approximation method originally introduced by Robbins and Monro \cite{R-M}.
Let $M(x)$ be a fixed function such that the equation $M(x)=\alpha$ has a unique solution $x=\theta$, where $\alpha$ is a given constant. At each $x$, the function $M(x)$ can be measured with random error $U(x)$,
$$
Y(x)=M(x)+U(x),\ \ \ \ \ee U(x)=0.
$$
In other words, $M(x)$ represents a regression function,
$$
M(x)=\int_{-\infty}^\infty y dF_x(y),
$$
where $F_x(y)$ is an unknown family of conditional distribution functions depend on the real parameter $x$.
The Robbins-Monro procedure for finding the root $\theta$ is defined as follows. Let $(a_n)_{ n\ge 0}$ be a sequence of positive numbers satisfying
$$
\sum_{n=0}^\infty a_n^2<\infty,\ \ \ \sum_{n=0}^\infty a_n=\infty.
$$
Given an initial guess $X_0$, define the sequence $(X_n)_{n\ge 0}$ by the recursion
\beq\label{R-M-2}
X_{n+1}=X_n+a_n(\alpha- Y_n),
\deq
where
$$
Y_n=Y(X_n)=M(X_n)+U_n,\ \ \ \ U_n=U(X_n)
$$
represents the measurement at the $n$-th iteration. Robbins and Monro \cite{R-M} proved that $X_n$ convergence to $\theta$ in probability. For further details, see \cite{B}, \cite{G-K}, \cite{G}, \cite{M-R}, \cite{R-D}, \cite{V} and \cite{W-72}.

In this article, we focus on a generalized stochastic approximation algorithm that builds upon the Robbins-Monro procedure. Let $(X_n)_{n\ge 0}$ be a stochastic process taking values in $\rr$, adapted to the filtration $(\FF_n)_{n\ge 0}$, and  satisfying the recursion:
\beq\label{sa-0}
X_{n+1}-X_n=\frac{b}{n+1}\big[g(X_n)+U_{n+1}\big],\  \ X_0=x_0\in\rr,
\deq
where $b>0, g:\rr\to\rr$ is a  function, and $\left(U_n\right)_{n\ge 1}$ is a sequence of bounded martingale differences adapted to the filtration $\left(\FF_n\right)_{n\ge 1}$. It is well known that the characteristics of the error term $U_{n}$ significantly influence the convergence behavior and rate of the algorithm (\ref{sa-0}). For identically distributed errors $(U_n)_{n\ge 1}$, Sacks \cite{S} showed that $\sqrt{n} \left(X_n-\theta\right)$ converges in distribution to a normal distribution. When $(U_n)_{n\ge 1}$ is a sequence of independent and identically distributed (i.i.d.) random variables, Dippon \cite{D-04} obtained (weak) higher order representations which involve sums of
linear, quadratic and cubic forms of the observation errors; Dippon \cite{D} also proved a limit theorem of Berry-Esseen type; Lai and
Robbins \cite{L-R} derived the law of the iterated logarithm for the recursion (\ref{sa-0}).  Moreover, when $(U_n)_{n\ge 1}$ is a sequence of uniformly bounded and i.i.d. random variables, Miao and Dong \cite{M-D} were the first to established the moderate deviation principle for the stochastic approximation algorithm (\ref{sa-0}).

In the case where $(U_n)_{n\ge 1}$ is a sequence of bounded random variables, for the recursion (\ref{sa-0}), Renlund \cite{R-1} proved that $X_n$ converges almost surely to a stable point $x^*$ of $g$,  where the stable point is defined in Section 2 of this paper.  In a subsequent work, Renlund \cite{R-2} established that $\sqrt{n} (X_n-x^*)$ is asymptotically normally distributed. Regarding large-scale fluctuations,  Shi et al. \cite{S-M} provided upper bounds for the large deviation
probabilities $\pp (|X_n-x^*|>\varepsilon)$ for $\varepsilon>0$ and sufficiently large $n$. Furthermore, Shi et al. \cite{S-M-1} obtained the law of the iterated logarithm for $X_n-x^*$.

Building on these results, we now turn to the moderate deviation principle for the stochastic approximation algorithm (\ref{sa-0}), where $(U_n)_{n\ge 1}$ is a sequence of bounded martingale differences.
Both moderate deviation and large deviation principles are essential in statistical theory.   While large deviation principles  provide precise asymptotic estimations corresponding to the law of large numbers, moderate deviation principles  refine the fluctuations associated with  the central limit theorem and the law of the iterated logarithm.
To illustrate this, let $(Y_n)_{n\ge 1}$ be a sequence of i.i.d. random variables with mean $\mu$ and  variance $\sigma^2_1$. The strong law of large numbers guarantees that the sample mean $\bar{Y}_n:=\frac{\sum_{k=1}^n Y_k}{n}$  converges almost surely to $\mu $.  Additionally, the central limit theorem characterizes its distributional fluctuations as $\frac{\bar{Y}_n-\mu}{\sigma_1 /\sqrt{n}}\xrightarrow{d} \nn(0,1)$. We are interested in the asymptotic behavior of the probability:
\beq\label{mdp-1}
\pp\left(\bar{Y}_n-\mu\in A_n\right),
\deq
where $A_n$ is a Borel subset of $\rr$  that represents the deviation of $\bar{Y}_n$ from $\mu$. The choice of $A_n$ determines the probabilistic regime.
Specifically, the case where $A_n=A$ for all $n$ corresponds to the theory of large deviations,  while $A_n=\frac{1}{\sqrt{n}}A$ is governed by the central limit theorem.
When $A_n=(\frac{b_n}{\sqrt{n}})A$ for a sequence of positive numbers $(b_n)_{n\ge 1}$ that satisfies $b_n\rightarrow\infty$ and  $\frac{b_n}{\sqrt{n}}\rightarrow0$, (\ref{mdp-1}) falls within the scope of  moderate deviations.  Consequently, moderate deviation provides an intermediate estimation between the central limit theorem and the large deviation.  For a more thorough discussion of large deviation theory and its applications, refer to Dembo
and Zeitouni \cite{D-Z}.

In this paper, we establish the moderate deviation principle for the stochastic approximation process (\ref{sa-0}). A key step in the proof is deriving the exponential inequality for the stochastic approximation process (\ref{sa-0}) and the moderate deviation principle for weighted sums of  bounded martingale differences.
These two results are crucial for proving the main theorem, and for clarity, we present them as separate theorems.
Our main results are presented in Section 2. The detailed proofs are provided in Section 3.
Throughout the paper, $C$ denotes a constant that may vary depending on the context.

\section{Main results}
Let $(X_n)_{n\ge 0}$ be a stochastic process taking values in $\rr$ and satisfying the stochastic approximation algorithm  (\ref{sa-0}). Denote by $g^{'}(x)$ the first derivative of the function $g$ with respect to $x$, and $g^{''}(x)$  the second derivative of the function $g$ at point $x$.  We assume the following conditions:
\begin{enumerate}[$(1)$]
\item  For any $n\ge 0$,
\beq\label{sa-2}
\ee \left[U_{n+1}|\FF_{n}\right]=0\ \text{and} \ |U_{n}|\le K_u.
\deq

\item  For any $n\ge 0$,
\beq\label{sa-6}
\ee\left[\left( U_{n+1}\right)^2\Big|\FF_{n}\right] = \sigma^2 \ \ a.s.
\deq

\item  For any $x\in \rr$,
\beq\label{sa-1}
\left|g^{''}(x)\right|\le K_a.
\deq

\item  For any $x\in \rr$, the function $g$ has a unique stable point $x^*$ (i.e., $g(x^*) = 0$ and $g'(x^*) < 0$) such that
\beq\label{g}
K_1\left|x-x^*\right|\le \left|g\left(x\right)\right|\le K_2\left|x-x^*\right|.
\deq

    \end{enumerate}
Here, $K_1, K_2, K_u, K_a, \sigma$ are positive constants.

\begin{rem}
Let $x^{*}$ be the unique stable point of the function $g$. Then, there exists a sufficiently small $\varepsilon>0$ such that
\beq\label{l1}
g(x)>0,\ \text{for}\ x\in(x^{*}-\varepsilon, x^{*})\ \ \ \text{and}\ \ \ g(x)<0, \ \text{for}\  x\in(x^{*}, x^{*}+\varepsilon),
\deq
which implies that
$$
(x-x^{*})g(x)\le0,\ \ x\in(x^{*}-\varepsilon, x^{*}+\varepsilon).
$$
\end{rem}

We now state an exponential inequality for the process $(X_n)_{n\ge 0}$, which plays a key role in proving the moderate deviation principle for this process.

\begin{thm}\label{lem-5}{\bf (Exponential inequality)}
Let $(X_n)_{n\ge 0}$ be a stochastic process taking values in $\rr$ and satisfying the stochastic approximation algorithm  (\ref{sa-0}). Assume that the conditions (\ref{sa-2}) and (\ref{g}) hold.
Then, for any $\varepsilon>0$, there exist a constant $\delta\in (0,1)$ such that, for all sufficiently large $n$, we have
$$
\pp\left(\left|X_{n+1}-x^*\right|\ge \varepsilon\right)\le 2 \exp\left(- \frac{C \varepsilon^2\delta n}{1-\delta}\right)\left(1+\frac{1}{1-\exp\left(- \frac{C \varepsilon^2  }{1-\delta}\right)}\right).
$$
Here $\delta$ can be chosen to satisfy
$$
0<\delta<\exp\left(-\frac{2(F+\varepsilon)}{b K_1 \varepsilon}\right),
$$
where $F=F\left(b, K_1, K_u\right)$ is a positive constant.
\end{thm}

\begin{rem}
Even if $(U_n)_{n\ge 1}$ is not a sequence of martingale differences, Theorem \ref{lem-5} still holds. Specifically, our proof requires only that $(U_n)_{n\ge 1}$ be bounded and that condition  (\ref{g}) is satisfied.
\end{rem}

The next result concerns the moderate deviation principle for  weighted sums of bounded martingale differences, which is crucial for proving the moderate deviation principle of the stochastic process $(X_n)_{n\ge 0}$.

\begin{thm}\label{lem-6}{\bf (Moderate deviation principle for weighted sums of  bounded martingale differences)}
Let $\{U_n, \FF_n\}_{n\ge 1}$ be a sequence of \ $\rr$-valued, bounded martingale differences, satisfying the conditions (\ref{sa-2}) and (\ref{sa-6}). Let $(b_n)_{n\ge 1}$ be a sequence of positive numbers such that $b_n\rightarrow\infty$ and  $\frac{b_n}{\sqrt{n}}\rightarrow0$.
For any constants $b$ and $\alpha_1$ satisfying $b\alpha_1<-1$, and for any $r>0$, we have
\beq\label{lem-6-0}
\lim_{n\rightarrow\infty} \frac{1}{b^2_n} \log \pp \left(h_n\left|b\sum_{k=0}^n \frac{1}{k+1}\beta_{k+1}^n\left(b\alpha_1\right)U_{k+1}\right|>r b_n\right)=-\frac{r^2}{2 \sigma^2},
\deq
where $h_n:=\left(b^2\sum_{k=0}^n \left(k+1\right)^{-2} \left(\beta_{k+1}^n(b \alpha_1)\right)^2\right)^{-\frac{1}{2}}$ and $\beta_{k}^n\left(b \alpha_1\right):=\prod_{j=k}^n \left(1+\frac{b \alpha_1}{j+1}\right)$.
\end{thm}

\begin{rem}
Eichelsbacher and  L\"{o}we \cite{E-M} established the moderate deviation principle for bounded martingale differences. Our result, which extends this to weighted bounded martingale differences, further generalizes their conclusion.
\end{rem}

The main result of this paper establishes the moderate deviation principle for the stochastic process $(X_n)_{n\ge 0}$.

\begin{thm}\label{thm-sa1}{\bf(Moderate deviation principle for stochastic approximation $(X_n)_{n\ge 0}$)}
Let $(X_n)_{n\ge 0}$ be a stochastic process taking values in $\rr$ and satisfying the stochastic approximation algorithm  (\ref{sa-0}). Assume that the conditions (\ref{sa-2})-(\ref{g}) hold, and that $bg^{'}(x^*)<-1$.
Let $(b_n)_{n\ge 1}$ be a sequence of positive numbers such that $b_n\rightarrow\infty$ and  $b_n=O(n^{\frac{1}{2(1+\gamma)}})$ for a constant $\gamma>0$.
Then, for any $r>0$, we have
$$
\lim_{n\rightarrow \infty} \frac{1}{b^2_n} \log \pp \left( h_n\left|X_{n+1}-x^*\right|>r b_n\right)=-\frac{r^2}{2 \sigma^2},
$$
where
$$
h_n:=\left(b^2\sum_{k=0}^n (k+1)^{-2} \left(\beta_{k+1}^n\left(b g^{'}\left(x^*\right)\right)\right)^2\right)^{-\frac{1}{2}}
$$
and
$$
\beta_{k}^n\left(b g^{'}\left(x^*\right)\right):=\prod_{j=k}^n \left(1+\frac{b g^{'}\left(x^*\right)}{j+1}\right).
$$
\end{thm}

\begin{rem} Under the conditions (\ref{sa-2}), (\ref{g}) and the additional assumption
\beq\label{sa-4}
\left|\ee\left[\frac{b}{n+1}U_{n+1}\  \Big|\ \FF_n\right]\right|\le \frac{K_e}{n^2},
\deq
where $K_e$ is a positive constant,
Renlund \cite{R-1} proved that $X_n$ converges almost surely to a stable point of $g$ when $X_n$ takes values in $[0,1]$.
\end{rem}

\begin{rem}
When  the sequence $(U_n)_{n\ge 1}$  is relaxed to bounded i.i.d. random variables with $\ee[U_n]=0$ and $\ee[U_n^2]=\sigma^2$,  Miao and Dong \cite{M-D} established the moderate deviation principle for the stochastic approximation algorithm (\ref{sa-0}) under conditions (\ref{sa-1}) and (\ref{g}), which is consistent with Theorem \ref{thm-sa1}.
\end{rem}

\section{Proofs of main results}

\begin{lem}\label{lem-1}
Let $\beta_{k}^n(c):=\prod_{j=k}^n \left(1+\frac{c}{j+1}\right)$, where $c<0$. We have
\begin{enumerate}[$(i)$]
\item For any $n\ge 1$ and $\left(-2c-1 \right)\vee 1 \le k\le n$, we have
\beq
\exp\left(-\frac{c^2}{\left(-2c-1 \right)\vee 1}\right) \left(\frac{n+1}{k}\right)^c\le \beta_{k}^n(c)\le\left(\frac{n}{k+1}\right)^c.
\deq

\item For $c<-1/2$, we have
\beq\label{lem1}
\lim_{n\rightarrow\infty} \frac{\left(\sum_{k=0}^n (k+1)^{-2} \left(\beta_{k}^n(c)\right)^2\right)^{-1}}{(-2c-1)n}=1.
\deq

\end{enumerate}
\end{lem}
\begin{proof}
From the inequalities
$$
1+x\le e^{x} \ \ \text{and} \ \ 1-x\ge e^{-x-x^2},\ \ \text{for all}\ \  0<x<1/2,
$$
we can get that for any $n\ge 1$ and $\left(-2c-1 \right)\vee 1 \le k\le n$,
$$
\exp\left(\sum_{j=k}^n \frac{c}{j+1}- \sum_{j=k}^n\frac{c^2}{(j+1)^2}\right)\le \beta_{k}^n(c)\le \exp\left(\sum_{j=k}^n \frac{c}{j+1}\right).
$$
By using the elementary inequality we know that
$$
\sum_{j=k}^n\frac{c^2}{(j+1)^2}\le \frac{c^2}{k}\le \frac{c^2}{\left(-2c-1 \right)\vee 1},
$$
then we have
$$
\exp\left(-\frac{c^2}{\left(-2c-1 \right)\vee 1}\right)\exp\left(\sum_{j=k}^n \frac{c}{j+1}\right)\le\beta_{k}^n(c)\le \exp\left(\sum_{j=k}^n \frac{c}{j+1}\right).
$$
For $k>1$, using the inequalities
$$
\ln n-\ln (k+1)\le \sum_{i=k}^n \frac{1}{i+1}\le \ln (n+1)- \ln k,
$$
we have, for any $n\ge 1$ and $\left(-2c-1 \right)\vee 1 \le k\le n$,
\beq\label{lem-1-1}
\exp\left(-\frac{c^2}{\left(-2c-1 \right)\vee 1}\right) \left(\frac{n+1}{k}\right)^c\le \beta_{k}^n(c)\le\left(\frac{n}{k+1}\right)^c.
\deq
From (\ref{lem-1-1}), we have that for any $\varepsilon>0$ and $c<-\frac{1}{2}$, there exists a positive constant $N_0$ such that
\beq\label{lem-1-2}
1-\varepsilon\le (-2c-1)n \sum_{k=N_0}^n (k+1)^{-2} \left(\beta_{k}^n(c)\right)^2\le 1+\varepsilon.
\deq
Furthermore, using (\ref{lem-1-1}) again, for $\left(-2c-1 \right)\vee 1 \le k\le N_0$, we have
\beq\label{lem-1-3}
n \left(\beta_{k}^n(c)\right)^2\rightarrow 0 \ \ \text{as} \ \ n\rightarrow\infty.
\deq
For the case $0\le k\le \left(-2c-1 \right)\vee 1$, noting that
$$
n \left(\beta_{k}^n(c)\right)^2=n \prod_{j=k}^{\left(-2c-1 \right) \vee 1} \left(1+\frac{c}{j+1}\right)^2 \prod_{j=\left(-2c-1\right)\vee 1+1}^n \left(1+\frac{c}{j+1}\right)^2,
$$
using (\ref{lem-1-1}) again we can get that the limit (\ref{lem-1-3}) also holds. Hence, from (\ref{lem-1-2}) and (\ref{lem-1-3}), (\ref{lem1}) can be proved.
\end{proof}

\begin{rem}\label{rem-1}
From the proof of the Lemma \ref{lem-1} we know that
$$
\left(\sum_{k=0}^n (k+1)^{-2} \left(\beta_{k}^n(c)\right)^2\right)^{-1}= O(n),\  \ \ n\rightarrow\infty.
$$
\end{rem}

\begin{lem}\label{lem-2}\cite[Theorem 3.13]{M-C} Let $\{Y_k, \FF_k\}_{k\ge 1}$ be a martingale difference sequence with $a_k\le Y_k\le b_k$ for each $k$, for suitable constants $a_k, b_k$. Then for any $t\ge 0$,
$$
\pp\left(\left|\sum_{k=1}^n Y_k\right|\ge t\right)\le 2\exp\left(-\frac{2t^2}{\sum_{k=1}^n\left(b_k-a_k\right)^2}\right).
$$
\end{lem}

\begin{lem}\label{lem-4}
Let $(X_n)_{n\ge 0}$ be a stochastic process taking values in $\rr$ and satisfying the stochastic approximation algorithm  (\ref{sa-0}).
Assume that  the conditions (\ref{sa-2}) and (\ref{g}) hold.
Then, for any $\varepsilon>0$ and $F>0$, there exists a positive constant $\delta\in(0,1)$ such that for all sufficiently large $n$, we have
\begin{align*}
\pp\left(\left|X_{n+1}-x^*\right|\ge \varepsilon\right)&\le \pp\Big(\left|X_{[\delta n]}-x^*\right|>F\Big)+
\pp\left(b \left|\sum_{i=[\delta n]}^n \frac{U_{i+1}}{i+1}\right|\ge 2\varepsilon\right)\\
&\ \ \   +\sum_{k=[\delta n]}^n\pp\left(b \left|\sum_{i=k+1}^n \frac{U_{i+1}}{i+1}\right|\ge \frac{\varepsilon}{4}\right).
\end{align*}
Here, $[\delta n]$ denotes the integer part of $\delta n$. The constant $\delta$ can be chosen to satisfy
$$
0<\delta<\exp\left(-\frac{2(F+\varepsilon)}{b K_1 \varepsilon}\right),
$$
where $K_1$ is defined in (\ref{g}).
\end{lem}

\begin{proof} {\bf Step 1:} Firstly, we shall give the inequality of $\dsp\pp\Big(X_{n+1}-x^*>\varepsilon\Big)$.
For any $\varepsilon>0$ and $\delta\in(0,1)$, let us define the following events:
$$
B^{(n)}(\varepsilon,\delta)=B(\varepsilon,\delta)=\left\{X_{[\delta n]}-x^*\ge \frac{\varepsilon}{2}, \cdots, X_n-x^*\ge\frac{\varepsilon}{2}, X_{n+1}-x^*\ge \varepsilon\right\},
$$
$$
A^{(n)}_0(\varepsilon)=A_0(\varepsilon)=\left\{X_1-x^{*}\ge \frac{\varepsilon}{2}, \cdots, X_n-x^*\ge \frac{\varepsilon}{2}, X_{n+1}-x^*\ge\varepsilon\right\},
$$
$$
A^{(n)}_k(\varepsilon)=A_k(\varepsilon)=\left\{X_k-x^*<\frac{\varepsilon}{2}, X_{k+1}-x^*\ge \frac{\varepsilon}{2}, \cdots, X_n-x^*\ge \frac{\varepsilon}{2}, X_{n+1}-x^*\ge\varepsilon\right\},
$$
where $k=1, 2, \cdots, n$. Hence, for any $F>0$, we can get that
\beq\label{p-2}
\aligned
&\pp\Big(X_{n+1}-x^*>\varepsilon\Big)\\
=&\pp\Big(X_{n+1}-x^*>\varepsilon, B(\varepsilon,\delta)\Big)+\pp\Big(X_{n+1}-x^*>\varepsilon,\left(B(\varepsilon,\delta)\right)^c\Big)\\
\le& \pp\Big(X_{[\delta n]}-x^*>F\Big)+\pp\Big(X_{[\delta n]}-x^*\le F, B(\varepsilon,\delta)\Big)+\pp\Big(X_{n+1}-x^*>\varepsilon,\left(B(\varepsilon,\delta)\right)^c\Big).
\endaligned
\deq
From (\ref{g})  and (\ref{l1})  we know that
\beq\label{g1}
-K_2(x-x^*)\le g(x)\le-K_1(x-x^*)\ \ \text{if} \ \ x>x^*.
\deq
For any $\varepsilon>0$ and $F>0$, we can choose
$$
\delta=\delta(\varepsilon, F, b, K_1)<\exp\left(-\frac{2(F+\varepsilon)}{b K_1 \varepsilon}\right)
$$
small enough such that for all $n$ large enough,
\beq\label{p-6}
F+b\sum_{i=[\delta n]}^n \frac{-K_1\frac{\varepsilon}{2}}{i+1}<-\varepsilon.
\deq
Hence, from (\ref{sa-0}),  (\ref{g1}) and  (\ref{p-6}), we have
\begin{align}\label{p-3}
&\pp\Big(X_{[\delta n]}-x^*\le F, B(\varepsilon,\delta)\Big)\nonumber\\
=&\pp\left(X_{[\delta n]}-x^*\le F, X_{[\delta n]}-x^*\ge \frac{\varepsilon}{2}, \cdots, X_n-x^*\ge\frac{\varepsilon}{2}, X_{n+1}-x^*\ge \varepsilon \right)\nonumber\\
=& \pp\Bigg(\varepsilon\le X_{n+1}-x^*=X_{[\delta n]}-x^*+ b\sum_{i=[\delta n]}^n \frac{g(X_i)+U_{i+1}}{i+1},\nonumber\\
&\ \ \ \ \ \ \ \ \ \ \ \ \ \ \ \ \ X_{[\delta n]}-x^*\le F, X_{[\delta n]}-x^*\ge \frac{\varepsilon}{2}, \cdots, X_n-x^*\ge\frac{\varepsilon}{2}\Bigg)\nonumber\\
\le& \pp\Bigg(\varepsilon\le X_{n+1}-x^*=X_{[\delta n]}-x^*+ b\sum_{i=[\delta n]}^n \frac{-K_1\frac{\varepsilon}{2}+U_{i+1}}{i+1},\nonumber\\
&\ \ \ \ \ \ \ \ \ \ \ \ \ \ \ \ \ X_{[\delta n]}-x^*\le F, X_{[\delta n]}-x^*\ge \frac{\varepsilon}{2}, \cdots, X_n-x^*\ge\frac{\varepsilon}{2}\Bigg)\nonumber\\
\le& \pp\left(\varepsilon\le  F+b\sum_{i=[\delta n]}^n \frac{-K_1\frac{\varepsilon}{2}+U_{i+1}}{i+1}\right)\nonumber\\
\le& \pp\left(b \sum_{i=[\delta n]}^n \frac{U_{i+1}}{i+1}\ge 2\varepsilon\right).
\end{align}
Next we estimate the probability $\pp\Big(X_{n+1}-x^*>\varepsilon,\left(B(\varepsilon,\delta)\right)^c\Big)$. Firstly, we know that
\begin{align*}
\pp\Big(X_{n+1}-x^*>\varepsilon,\left(B(\varepsilon,\delta)\right)^c\Big)\le \sum_{k=[\delta n]}^n \pp(A_k(\varepsilon)).
\end{align*}
From (\ref{sa-0}), (\ref{sa-2}) and (\ref{g}) we can get that
$$
\aligned
X_{k+1}-x^{*}=&X_k-x^{*}+\frac{b}{k+1} \left(g(X_k)+ U_{k+1}\right)\\
\le& X_k-x^{*}+\frac{b K_u}{k+1}+\max\left\{0, \frac{-K_2b\left( X_k-x^*\right)}{k+1}\right\}.
\endaligned
$$
Then for any $\varepsilon>0$, there exists a positive constant $k_1$, such that for all $k>k_1$, we have
\beq\label{X}
\left\{X_k-x^{*}\le \frac{1}{2} \varepsilon\right\}\subset \left\{X_{k+1}-x^{*}\le\frac{3}{4}\varepsilon\right\}.
\deq
Hence for all $n$ large enough and $k\ge [\delta n]$, from (\ref{sa-0}), (\ref{g1}) and (\ref{X}), we have
\begin{align*}
\pp(A_k(\varepsilon))&=\pp\left(X_k-x^*<\frac{\varepsilon}{2}, X_{k+1}-x^*\ge \frac{\varepsilon}{2}, \cdots, X_n-x^*\ge \frac{\varepsilon}{2}, X_{n+1}-x^*\ge\varepsilon\right)\\
&=\pp\Bigg(X_k-x^*<\frac{\varepsilon}{2},  X_{k+1}-x^*+b\sum_{i=k+1}^n \frac{g(X_i)+U_{i+1}}{i+1} \ge\varepsilon,\ \\
&\ \ \ \ \ \ \ \ \ \ \ \ \ \ \ \   X_{k+1}-x^*\ge \frac{\varepsilon}{2}, \cdots, X_n-x^*\ge \frac{\varepsilon}{2}\Bigg)\\
&\le \pp\Bigg(X_k-x^*<\frac{\varepsilon}{2},  X_{k+1}-x^*+b \sum_{i=k+1}^n \frac{-K_1\frac{\varepsilon}{2}+U_{i+1}}{i+1} \ge\varepsilon,\ \\
&\ \ \ \ \ \ \ \ \ \ \ \ \ \ \ \   X_{k+1}-x^*\ge \frac{\varepsilon}{2}, \cdots, X_n-x^*\ge \frac{\varepsilon}{2}\Bigg)\\
&\le \pp\left(X_k-x^*<\frac{\varepsilon}{2}, X_{k+1}-x^*+b \sum_{i=k+1}^n \frac{U_{i+1}}{i+1} \ge \varepsilon \right)\\
&\le \pp\left(X_{k+1}-x^{*}\le\frac{3}{4}\varepsilon, X_{k+1}-x^*+b \sum_{i=k+1}^n \frac{U_{i+1}}{i+1} \ge\varepsilon\right)\\
&\le \pp\left(b \sum_{i=k+1}^n \frac{U_{i+1}}{i+1}\ge \frac{\varepsilon}{4}\right).
\end{align*}
Then we have
\beq\label{lem4-1}
\pp\Big(X_{n+1}-x^*>\varepsilon,\left(B(\varepsilon,\delta)\right)^c\Big)\le \sum_{k=[\delta n]}^n \pp(A_k(\varepsilon))\le \sum_{k=[\delta n]}^n\pp\left(b \sum_{i=k+1}^n \frac{U_{i+1}}{i+1}\ge\frac{\varepsilon}{4}\right).
\deq
Hence from (\ref{p-2}), (\ref{p-3}) and (\ref{lem4-1}), for any $\varepsilon>0$ and $F>0$, there exists a positive constant $\delta\in(0,1)$ such that for all sufficiently large $n$, we have
\beq\label{zong-1}
\aligned
\pp\left(X_{n+1}-x^*\ge \varepsilon\right)&\le \pp\Big(X_{[\delta n]}-x^*>F\Big)+
\pp\left(b \sum_{i=[\delta n]}^n \frac{U_{i+1}}{i+1}\ge 2\varepsilon\right)\\
&\ \ \   +\sum_{k=[\delta n]}^n\pp\left(b \sum_{i=k+1}^n \frac{U_{i+1}}{i+1}\ge \frac{\varepsilon}{4}\right).
\endaligned
\deq

{\bf Step 2:} Next we shall obtain the inequality of $\dsp\pp\Big(X_{n+1}-x^*<-\varepsilon\Big)$.
For any $\varepsilon>0$ and $\delta\in(0,1)$, let us define the following events:
$$
C^{(n)}(\varepsilon,\delta)=C(\varepsilon,\delta)=\left\{X_{[\delta n]}-x^*\le -\frac{\varepsilon}{2}, \cdots, X_n-x^*\le-\frac{\varepsilon}{2}, X_{n+1}-x^*\le -\varepsilon\right\},
$$
$$
D^{(n)}_0(\varepsilon)=D_0(\varepsilon)=\left\{X_1-x^{*}\le -\frac{\varepsilon}{2}, \cdots, X_n-x^*\le -\frac{\varepsilon}{2}, X_{n+1}-x^*\le-\varepsilon\right\},
$$
$$
\aligned
D^{(n)}_k(\varepsilon)=&D_k(\varepsilon)\\
=&\left\{X_k-x^*>-\frac{\varepsilon}{2}, X_{k+1}-x^*\le -\frac{\varepsilon}{2}, \cdots, X_n-x^*\le -\frac{\varepsilon}{2}, X_{n+1}-x^*\le-\varepsilon\right\},
\endaligned
$$
where $k=1, 2, \cdots, n$. Hence, for any $F>0$, we can get that
\beq\label{p-2-1}
\aligned
&\pp\Big(X_{n+1}-x^*<-\varepsilon\Big)\\
=&\pp\Big(X_{n+1}-x^*<-\varepsilon, C(\varepsilon,\delta)\Big)+\pp\Big(X_{n+1}-x^*<-\varepsilon,\left(C(\varepsilon,\delta)\right)^c\Big)\\
\le& \pp\Big(X_{[\delta n]}-x^*<-F\Big)+\pp\Big(X_{[\delta n]}-x^*\ge- F, C(\varepsilon,\delta)\Big)+\pp\Big(x_{n+1}-x^*<-\varepsilon,\left(C(\varepsilon,\delta)\right)^c\Big).
\endaligned
\deq
From (\ref{g})  and (\ref{l1})  we know that
\beq\label{g2}
-K_1(x-x^*)\le g(x)\le -K_2(x-x^*)\ \ \text{if} \ \ x<x^*.
\deq
For any $\varepsilon>0$ and $F>0$, we can choose
$$
\delta=\delta(\varepsilon, F, b, K_1)<\exp\left(-\frac{2(F+\varepsilon)}{b K_1 \varepsilon}\right)
$$
small enough such that for all $n$ large enough,
\beq\label{p-61}
F+b\sum_{i=[\delta n]}^n \frac{-K_1\frac{\varepsilon}{2}}{i+1}<-\varepsilon.
\deq
Hence, from (\ref{sa-0}),  (\ref{g2}) and  (\ref{p-61}), we have
\begin{align}\label{p-3-1}
&\pp\Big(X_{[\delta n]}-x^*\ge -F, C(\varepsilon,\delta)\Big)\nonumber\\
=&\pp\left(X_{[\delta n]}-x^*\ge -F, X_{[\delta n]}-x^*\le -\frac{\varepsilon}{2}, \cdots, X_n-x^*\le-\frac{\varepsilon}{2}, X_{n+1}-x^*\le -\varepsilon \right)\nonumber\\
=& \pp\Bigg( X_{n+1}-x^*=X_{[\delta n]}-x^*+b\sum_{i=[\delta n]}^n \frac{g(X_i)+U_{i+1}}{i+1}\le-\varepsilon,\nonumber\\
&\ \ \ \ \ \ \ \ \ \ \ \ \ \ \ \ \ X_{[\delta n]}-x^*\ge -F, X_{[\delta n]}-x^*\le -\frac{\varepsilon}{2}, \cdots, X_n-x^*\le-\frac{\varepsilon}{2}\Bigg)\nonumber\\
\le& \pp\Bigg( X_{n+1}-x^*=X_{[\delta n]}-x^*+b\sum_{i=[\delta n]}^n \frac{K_1 \frac{\varepsilon}{2}+U_{i+1}}{i+1}\le -\varepsilon,\nonumber\\
&\ \ \ \ \ \ \ \ \ \ \ \ \ \ \ \ \ X_{[\delta n]}-x^*\ge-F, X_{[\delta n]}-x^*\le -\frac{\varepsilon}{2}, \cdots, X_n-x^*\le-\frac{\varepsilon}{2}\Bigg)\nonumber\\
\le& \pp\left( -F+b\sum_{i=[\delta n]}^n \frac{K_1 \frac{\varepsilon}{2}+U_{i+1}}{i+1}\le-\varepsilon \right)\nonumber\\
\le& \pp\left(b\sum_{i=[\delta n]}^n \frac{U_{i+1}}{i+1}\le -2\varepsilon\right).
\end{align}
Next we estimate the probability $\pp\Big(X_{n+1}-x^*<-\varepsilon,\left(C(\varepsilon,\delta)\right)^c\Big)$. Firstly, we know that
\begin{align*}
\pp\Big(X_{n+1}-x^*<-\varepsilon,\left(C(\varepsilon,\delta)\right)^c\Big)\le \sum_{k=[\delta n]}^n \pp(D_k(\varepsilon)).
\end{align*}
 From (\ref{sa-0}), (\ref{sa-2}) and (\ref{g}) we can get that
$$
\aligned
X_{k+1}-x^{*}=&X_k-x^{*}+\frac{b}{k+1} \left(g(X_k)+U_{k+1}\right)\\
\ge& X_k-x^{*}-\frac{bK_u}{k+1}+\min\left\{0, \frac{-K_1 b \left( X_k-x^*\right)}{k+1}\right\}.
\endaligned
$$
Then for any $\varepsilon>0$, there exists a positive constant $k_2$, such that for all $k>k_2$, we have
\beq\label{X1}
\left\{X_k-x^{*}\ge- \frac{1}{2} \varepsilon\right\}\subset \left\{X_{k+1}-x^{*}\ge-\frac{3}{4}\varepsilon\right\}.
\deq
Hence for all $n$ large enough and $k\ge [\delta n]$, from (\ref{sa-0}), (\ref{g2}) and (\ref{X1}), we have
\begin{align*}
\pp(D_k(\varepsilon))&=\pp\left(X_k-x^*>-\frac{\varepsilon}{2}, X_{k+1}-x^*\le -\frac{\varepsilon}{2}, \cdots, X_n-x^*\le -\frac{\varepsilon}{2}, X_{n+1}-x^*\le-\varepsilon\right)\\
&=\pp\Bigg(X_k-x^*>-\frac{\varepsilon}{2},  X_{k+1}-x^*+b\sum_{i=k+1}^n \frac{g(X_i)+U_{i+1}}{i+1} \le-\varepsilon,\ \\
&\ \ \ \ \ \ \ \ \ \ \ \ \ \ \ \   X_{k+1}-x^*\le -\frac{\varepsilon}{2}, \cdots, X_n-x^*\le -\frac{\varepsilon}{2}\Bigg)\\
&\le \pp\Bigg(X_k-x^*>-\frac{\varepsilon}{2},  X_{k+1}-x^*+b\sum_{i=k+1}^n \frac{K_1 \frac{\varepsilon}{2}+U_{i+1}}{i+1} \le-\varepsilon,\ \\
&\ \ \ \ \ \ \ \ \ \ \ \ \ \ \ \   X_{k+1}-x^*\le -\frac{\varepsilon}{2}, \cdots, X_n-x^*\le -\frac{\varepsilon}{2}\Bigg)\\
&\le \pp\left(X_k-x^*>-\frac{\varepsilon}{2}, X_{k+1}-x^*+b\sum_{i=k+1}^n \frac{U_{i+1}}{i+1} \le -\varepsilon \right)\\
&\le \pp\left(X_{k+1}-x^{*}\ge-\frac{3}{4}\varepsilon, X_{k+1}-x^*+b\sum_{i=k+1}^n \frac{U_{i+1}}{i+1} \le-\varepsilon\right)\\
&\le \pp\left( \sum_{i=k+1}^n \frac{U_{i+1}}{i+1}\le -\frac{\varepsilon}{4}\right).
\end{align*}
Then we have
\beq\label{lem4-2}
\pp\Big(X_{n+1}-x^*<-\varepsilon,\left(C(\varepsilon,\delta)\right)^c\Big)\le \sum_{k=[\delta n]}^n \pp(D_k(\varepsilon))\le \sum_{k=[\delta n]}^n\pp\left(b \sum_{i=k+1}^n \frac{U_{i+1}}{i+1}\le-\frac{\varepsilon}{4}\right).
\deq
Hence from (\ref{p-2-1}), (\ref{p-3-1}) and (\ref{lem4-2})  for any $\varepsilon>0$ and $F>0$, there exists a positive constant $\delta\in(0,1)$ such that for all sufficiently large $n$, we have
\beq\label{zong-2}
\aligned
\pp\left(X_{n+1}-x^*\ge \varepsilon\right)&\le \pp\Big(X_{[\delta n]}-x^*<-F\Big)+
\pp\left(b \sum_{i=[\delta n]}^n \frac{U_{i+1}}{i+1}\le- 2\varepsilon\right)\\
&\ \ \   +\sum_{k=[\delta n]}^n\pp\left(b \sum_{i=k+1}^n \frac{U_{i+1}}{i+1}\le- \frac{\varepsilon}{4}\right).
\endaligned
\deq
Combining (\ref{zong-1}) and (\ref{zong-2}), we complete the proof.
\end{proof}

\begin{lem}\label{lem-3}
Let $(X_n)_{n\ge 0}$ be a stochastic process taking values in $\rr$ and satisfying the stochastic approximation algorithm  (\ref{sa-0}).
Assume that  the conditions (\ref{sa-2}) and (\ref{g}) hold. Then recursion $X_n$ is bounded, i.e. there exist a positive constant $C$ such that
$$
\left|X_{n+1}-x^*\right|\le C\left(|X_{0}-x^*| n^{-b K_1}+1\right).
$$
\end{lem}

\begin{proof}
From (\ref{sa-0}), (\ref{sa-2}), (\ref{g}) and Lemma \ref{lem-1} we know that
\begin{align*}
\left|X_{n+1}-x^*\right|&=\left|X_{n}-x^*+\frac{b}{n+1}g(X_n)+\frac{b}{n+1} U_{n+1}\right|\\
&=\left|\left(X_{n}-x^*\right)\left(1+\frac{b}{n+1}\frac{g(X_n)}{X_n-x^*}\right)+\frac{b}{n+1} U_{n+1}\right|\\
&=\left|\left(X_{0}-x^*\right) \prod_{k=0}^n \left(1+\frac{b g(X_k)}{(k+1)\left(X_{k}-x^*\right)}\right)\right. \\ &\ \ \ \ \left.+\sum_{i=0}^n \frac{bU_{i+1}}{i+1} \prod_{k=i+1}^n \left(1+\frac{b g(X_k)}{(k+1)\left(X_{k}-x^*\right)} \right)\right|\\
&\le \left|X_{0}-x^*\right| \prod_{k=0}^n \left|1-\frac{b K_1}{k+1}\right|+\sum_{i=0}^n \frac{bK_u}{i+1} \prod_{k=i+1}^n \left|1-\frac{b K_1}{k+1} \right|\\
&\le C |X_{0}-x^*| n^{-b K_1}+C  n^{-b K_1} \sum_{i=0}^n\frac{(i+2)^{bK_1}}{i+1}\\
&\le C\left(|X_{0}-x^*| n^{-b K_1}+1\right).
\end{align*}
\end{proof}

\begin{proof}[{\bf Proof of Theorem \ref{lem-5}}]
From Lemma \ref{lem-3}, we can choose $F=F(b, K_1, K_u)>0$ such that $\left|X_{n+1}-x^*\right|\le F$ for all $n$. Consequently, we  conclude that
$$
\pp\Big(\left|X_{[\delta n]}-x^*\right|>F\Big)=0.
$$
For any $\varepsilon>0$, we can choose $\delta=\delta(\varepsilon, F, b, K_1)$
small enough, such that for all $n$ large enough,
$$
F+b\sum_{i=[\delta n]}^n \frac{-K_1\frac{\varepsilon}{2}}{i+1}<-\varepsilon.
$$
Then, from Lemma \ref{lem-2} - Lemma \ref{lem-3},  for all large enough $n$, we have
\begin{align*}
&\ \ \ \pp\left(\left|X_{n+1}-x^*\right|\ge\varepsilon\right)\\
&\le
\pp\left(b \left|\sum_{i=[\delta n]}^n \frac{U_{i+1}}{i+1}\right|\ge 2\varepsilon\right)+\sum_{k=[\delta n]}^n\pp\left(b \left|\sum_{i=k+1}^n \frac{U_{i+1}}{i+1}\right|\ge\frac{ \varepsilon}{4}\right)\\
&\le 2\exp\left(- \frac{8\varepsilon^2}{b^2 K_u^2 \sum_{i=[\delta n]}^n \frac{1}{(i+1)^2}}\right)+2 \sum_{k=[\delta n]}^n \exp\left(- \frac{\varepsilon^2}{8b^2 K_u^2 \sum_{i=k+1}^n \frac{1}{(i+1)^2}}\right)\\
&\le 2 \exp\left(- \frac{C \varepsilon^2\delta n}{1-\delta}\right)+ 2 \sum_{k=[\delta n]}^n \exp\left(- \frac{C \varepsilon^2 }{k^{-1}-n^{-1}}\right)\\
&\le 2 \exp\left(- \frac{C \varepsilon^2\delta n}{1-\delta}\right)+ 2 \sum_{k=[\delta n]}^n \exp\left(- \frac{C \varepsilon^2 k }{1-\delta}\right)\\
&\le 2\exp\left(- \frac{C \varepsilon^2\delta n}{1-\delta}\right)+\frac{2\exp\left(- \frac{C \varepsilon^2 \delta n }{1-\delta}\right)}{1-\exp\left(- \frac{C \varepsilon^2  }{1-\delta}\right)}\\
&\le 2 \exp\left(- \frac{C \varepsilon^2\delta n}{1-\delta}\right)\left(1+\frac{1}{1-\exp\left(- \frac{C \varepsilon^2  }{1-\delta}\right)}\right).
\end{align*}
\end{proof}

\begin{proof}[{\bf Proof of Theorem \ref{lem-6}}]
In order to obtain claim (\ref{lem-6-0}), by the G\"{a}rtner-Ellis theorem \cite[Theorem 2.3.6]{D-Z}, we need to prove that the following
Cram\'{e}r functional holds: for any $\lambda \in \rr$,
\begin{align}\label{t-2}
\lim_{n\rightarrow\infty}\frac{1}{b^2_n} \log \ee\left[\exp\left(\lambda h_n b_n \sum_{k=0}^n \frac{b}{k+1}\beta_{k+1}^n\left(b\alpha_1\right)U_{k+1}\right)\right]=\frac{\lambda^2 \sigma^2}{2}.
\end{align}

Let $Y_1, \ldots, Y_n$ be independent and identically distributed  normal distributed random variables,  independent of the $(U_i)_{i\ge1}$, with $Y_i\sim N(0, \sigma^2)$, where $\sigma$ is defined as $(\ref{sa-6})$. For $0\le k\le n$,  let
\beq\label{a-3}
A_{k+1}:=h_n \frac{b}{k+1} \beta_{k+1}^n\left(b\alpha_1\right)  Y_{k+1},
\deq
then, by Gaussian integration and the definition of $h_n$, it is not difficult to see that
\begin{align}\label{a-4}
\ee \left[\exp\left(\lambda b_n \sum_{k=0}^n A_{k+1}\right)\right]&=\ee \left[\exp\left(\lambda b_n \sum_{k=0}^n h_n \frac{b}{k+1} \beta_{k+1}^n\left(b\alpha_1\right)  Y_{k+1}\right)\right]\nonumber\\
&=\exp\left(\frac{\lambda^2 b_n^2 \sigma^2 h_n^2 \sum_{k=0}^n\frac{b^2}{(k+1)^2} \left(\beta_{k+1}^n\left(b\alpha_1\right)\right)^2 }{2}\right)\nonumber\\
&=\exp\left(\frac{\lambda^2 b_n^2 \sigma^2 }{2}\right).
\end{align}
For $0\le k\le n$, let
\beq\label{a-1}
T_{k+1}:=h_n \frac{b}{k+1} \beta_{k+1}^n\left(b\alpha_1\right) U_{k+1},
\deq
then we know that
\beq\label{a-2}
h_n \sum_{k=0}^n \frac{b}{k+1}\beta_{k+1}^n\left(b\alpha_1\right)U_{k+1}:= \sum_{k=0}^n T_{k+1}.
\deq
If we can prove that
\beq\label{a-5}
\lim_{n\rightarrow\infty} \frac{1}{b^2_n} \log \frac{\ee \left[\exp\left(\lambda b_n \sum_{k=0}^n T_{k+1}\right)\right]}{\ee \left[\exp\left(\lambda b_n \sum_{k=0}^n A_{k+1}\right)\right]}=0,
\deq
from (\ref{a-4}), it is easy to see that
$$
\lim_{n\rightarrow\infty}\frac{1}{b^2_n} \log \ee \left[\exp\left(\lambda h_n b_n \sum_{k=0}^n \frac{b}{k+1}\beta_{k+1}^n\left(b\alpha_1\right)U_{k+1}\right)\right]=\frac{\lambda^2 \sigma^2}{2}.
$$

In the following we mainly prove (\ref{a-5}).
Let
\beq\label{a-6}
Z_m:=h_n \sum_{j=0}^{m-1}  \frac{b}{j+1} \beta_{j+1}^n\left(b\alpha_1\right) U_{j+1}:= \sum_{j=0}^{m-1}T_{j+1}, \ Z_0=0
\deq
and
\beq\label{a-7}
W_m:=h_n\sum_{j=m+1}^n  \frac{b}{j+1} \beta_{j+1}^n\left(b\alpha_1\right) Y_{j+1}:=\sum_{j=m+1}^n A_{j+1},\ W_n=0.
\deq
From (\ref{a-3}) and (\ref{a-1}), we can get that
\begin{align}\label{b-1}
\frac{\ee \left[\exp\left(\lambda b_n \sum_{k=0}^n T_{k+1}\right)\right]}{\ee \left[\exp\left(\lambda b_n \sum_{k=0}^n A_{k+1}\right)\right]}
&=1+\frac{\ee\Big[\exp\left(\lambda b_n \sum_{k=0}^n T_{k+1}\right)\Big]-\ee \Big[\exp\left(\lambda b_n \sum_{k=0}^n A_{k+1}\right)\Big]}{\ee \Big[\exp\left(\lambda b_n \sum_{k=0}^n A_{k+1}\right)\Big]}\nonumber\\
&=1+\frac{\sum_{m=0}^ {n} Q_m}{\ee \Big[\exp\left(\lambda b_n \sum_{k=0}^n A_{k+1}\right)\Big]},
\end{align}
where
\begin{align}\label{Q}
Q_m:=&\ee \Big[\exp\left(\lambda b_n\left(W_m+Z_m+T_{m+1}\right)\right)\Big]\nonumber\\
&\ \ \  -\ee \Big[\exp\left(\lambda b_n \left(W_m+Z_m+A_{m+1}\right)\right)\Big].
\end{align}

Next we consider $\sum_{m=0}^n Q_m$. From (\ref{a-6}) and (\ref{a-7}), it is easy to see that $Z_m \in \FF_{m}$ and that  $Z_m$ and $W_m$ are independent of each other. Then, it is not difficult to see that
\begin{align*}
\ee &\Big[\exp\big(\lambda b_n \left(W_m+Z_m+T_{m+1}\right)\big)\Big]\\
=&\ee\Bigg[\ee \Big[\exp\big(\lambda b_n \left(W_m+Z_m+T_{m+1}\right)\big)\Big|\FF_{m}\Big]\Bigg]\\
=&\ee \Big[\exp(\lambda b_n W_m)\Big] \ee\Bigg[\exp(\lambda b_n Z_m) \ee \Big[\exp(\lambda b_n T_{m+1})\Big|\FF_{m}\Big]\Bigg]
\end{align*}
and
\begin{align*}
\ee &\Big[\exp\left(\lambda b_n \left(W_m+Z_m+A_{m+1}\right)\right)\Big]\\
=&\ee\Bigg[\ee \Big[\exp\left(\lambda b_n \left(W_m+Z_m+A_{m+1}\right)\right)\Big|\FF_{m}\Big]\Bigg]\\
=&\ee \Big[\exp(\lambda b_n W_m)\Big] \ee\Bigg[\exp(\lambda b_n Z_m) \ee \Big[\exp\left(\lambda b_n A_{m+1}\right)\Big|\FF_{m}\Big]\Bigg].
\end{align*}
Therefore, we can get that
\begin{align}\label{a-8}
Q_m
&=\ee \Big[\exp(\lambda b_n W_m)\Big] \ee\Big[\exp(\lambda b_n Z_m)R_m\Big],
\end{align}
where
\beq\label{a-9}
R_m:=\ee \Big[\exp(\lambda b_n T_{m+1})\Big|\FF_{m}\Big]-\ee \Big[\exp(\lambda b_n A_{m+1})\Big|\FF_{m}\Big].
\deq

Note that $\frac{b_n}{\sqrt{n}}\rightarrow 0$ as $n\rightarrow\infty$ and $b\alpha_1<-1 $. From Remark \ref{rem-1}, we have
\beq\label{h-lem}
h_n:=\left(b^2\sum_{k=0}^n (k+1)^{-2} \left(\beta_{k+1}^n\left(b \alpha_1\right)\right)^2\right)^{-\frac{1}{2}}=O(\sqrt{n}),\ \  n\rightarrow\infty.
\deq
Then from (\ref{a-1}) and Lemma \ref{lem-1}, we know that
\begin{align*}
&\ee \left[\exp\left(\lambda b_n \left|T_{m+1}\right|\right)\Big|\FF_{m}\right]\\
=& \ee \left[\exp\left(\lambda b_n  h_n \frac{b}{m+1} \beta_{m+1}^n(b\alpha_1) \left|U_{m+1}\right|\right)\Bigg|\FF_{m}\right]\\
\le& C \exp\left(\lambda b_n \sqrt{n} (m+1)^{-1-b\alpha_1} n^{b\alpha_1}\right)\\
\le& C \exp\left(\lambda b_n \sqrt{n} n^{-1-b\alpha_1} n^{b\alpha_1}\right)<\infty.
\end{align*}
Therefore, using the Taylor expansion, we can obtain from (\ref{sa-2}) and (\ref{sa-6})  that
\begin{align}\label{a-10}
& \ee \Big[\exp(\lambda b_n T_{m+1})\Big|\FF_{m}\Big]\nonumber\\
=&1+\lambda b_n  h_n \frac{b}{m+1} \beta_{m+1}^n(b\alpha_1)  \ee[ U_{m+1}|\FF_{m}]\nonumber\\
 &+\frac{\lambda^2 b_n^2 h^2_n }{2} \frac{b^2}{(m+1)^2} \left(\beta_{m+1}^n(b\alpha_1)\right)^2\ee[ U^2_{m+1}|\FF_{m}]+O\left(\lambda^3 b_n^3 h^3_n \frac{b^3}{(m+1)^3} \left(\beta_{m+1}^n(b\alpha_1)\right)^3\right) \nonumber\\
=&1+\frac{\lambda^2 b_n^2 h^2_n }{2} \frac{b^2}{(m+1)^2} \left(\beta_{m+1}^n(b\alpha_1)\right)^2\ee[ U^2_{m+1}|\FF_{m}]\nonumber\\
&\ \ \  +O\left(\lambda^3 b_n^3 h^3_n \frac{b^3}{(m+1)^3} \left(\beta_{m+1}^n(b\alpha_1)\right)^3\right)\nonumber\\
 =&1+\frac{\lambda^2 b_n^2 h^2_n }{2} \frac{b^2}{(m+1)^2} \left(\beta_{m+1}^n(b\alpha_1)\right)^2\sigma^2
 +O\left(\lambda^3 b_n^3 h^3_n \frac{b^3}{(m+1)^3} \left(\beta_{m+1}^n(b\alpha_1)\right)^3\right).
\end{align}
Likely, from (\ref{a-3}), (\ref{h-lem}) and Lemma \ref{lem-1}, we know that
\begin{align*}
&\ee \left[\exp\left(\lambda b_n \left| A_{m+1}\right|\right)\Big|\FF_{m}\right]\nonumber\\
=&\ee \left[\exp\left(\lambda b_n h_n \frac{b}{m+1} \beta_{m+1}^n(b\alpha_1)  \left|Y_{m+1}\right|\right)\right]\nonumber\\
\le&2 \exp\left(\frac{\lambda^2 b_n^2 \sigma^2 h_n^2 \frac{b^2}{(m+1)^2} \left(\beta_{m+1}^n(b\alpha_1)\right)^2 }{2}\right)\nonumber\\
\le &2 \exp\left(C  b_n^2 n (m+1)^{-2-2 b\alpha_1} n^{2 b\alpha_1} \right)\nonumber\\
\le & 2 \exp\left(C  b_n^2 n n^{-2-2 b\alpha_1} n^{2 b\alpha_1} \right) <\infty.
\end{align*}
Observing that $Y_i\sim N(0,\sigma^2)$, we can obtain from (\ref{a-3}) and the Taylor expansion that
\begin{align}\label{a-11}
&\ \ \ \ee \Big[\exp(\lambda b_n A_{m+1})\Big|\FF_{m}\Big]\nonumber\\
=&1+\lambda b_n  h_n \frac{b}{m+1} \beta_{m+1}^n(b\alpha_1)  \ee[ Y_{m+1}]\nonumber\\
&+\frac{\lambda^2 b_n^2 h^2_n }{2} \frac{b^2}{(m+1)^2} \left(\beta_{m+1}^n(b\alpha_1)\right)^2\ee[ Y^2_{m+1}]
 +O\left(\lambda^3 b_n^3 h^3_n \frac{b^3}{(m+1)^3} \left(\beta_{m+1}^n(b\alpha_1)\right)^3\right) \nonumber\\
=&1+\frac{\lambda^2 b_n^2 h^2_n }{2} \frac{b^2}{(m+1)^2} \left(\beta_{m+1}^n(b\alpha_1)\right)^2\ee[ Y^2_{m+1}]
 +O\left(\lambda^3 b_n^3 h^3_n \frac{b^3}{(m+1)^3} \left(\beta_{m+1}^n(b\alpha_1)\right)^3\right)\nonumber\\
=&1+\frac{\lambda^2 b_n^2 h^2_n }{2} \frac{b^2}{(m+1)^2} \left(\beta_{m+1}^n(b\alpha_1)\right)^2\sigma^2
 +O\left(\lambda^3 b_n^3 h^3_n \frac{b^3}{(m+1)^3} \left(\beta_{m+1}^n(b\alpha_1)\right)^3\right).
\end{align}
From (\ref{a-9}), (\ref{a-10}), and (\ref{a-11}), we can get that
\beq\label{a-12}
R_m= O\left(\lambda^3 b_n^3 h^3_n \frac{b^3}{(m+1)^3} \left(\beta_{m+1}^n(b\alpha_1)\right)^3\right).
\deq

Finally, we have to bound $ \ee\Big[\exp\left(\lambda b_n W_m \right)\Big] \ee\Big[\exp\left(\lambda b_nZ_m \right)\Big]$. On the one hand, by Gaussian integration, we can determine from (\ref{a-7}) that
\begin{align}\label{a-13}
\ee\big[\exp\left(\lambda b_n  W_m\right)\big]&=\ee\Bigg[\exp\left(\lambda b_n h_n  \sum_{j=m+1}^n  \frac{b}{j+1} \beta_{j+1}^n(b\alpha_1) Y_{j+1}\right)\Bigg]\nonumber\\
&=\exp\left(\frac{\lambda^2 b_n^2 h_n^2  \sigma^2}{2}\sum_{j=m+1}^n \frac{b^2}{(j+1)^2} \left(\beta_{j+1}^n(b\alpha_1)\right)^2 \right).
\end{align}
On the other hand, from  (\ref{a-6}), we can get that
\begin{align}\label{a-17}
\ee\Big[\exp\left(\lambda b_n  Z_m\right)\Big]&=\ee\Bigg[\exp\left(\lambda b_n  \sum_{j=0}^{m-1} T_{j+1}\right)\Bigg]\nonumber\\
&=\ee \left[\ee\left[\exp\left(\lambda b_n  \sum_{j=0}^{m-1} T_{j+1}\right)\Bigg|\FF_{m-1}\right]\right]\nonumber\\
&=\ee \left[\prod_{j=0}^{m-2}\exp\left(\lambda b_n  T_{j+1}\right) \ee \left[\exp\left(\lambda b_n T_m\right)\Big|\FF_{m-1}\right]\right].
\end{align}
Noting that $b\alpha_1<-1$, and applying (\ref{h-lem}), (\ref{a-10})  and Lemma \ref{lem-1} , we easily conclude that
\begin{align*}
&\ee \Big[\exp(\lambda b_n T_{m+1})\Big|\FF_{m}\Big]\\
=& 1+\frac{\lambda^2 b_n^2 h^2_n }{2} \frac{b^2}{(m+1)^2} \left(\beta_{m+1}^n(b\alpha_1)\right)^2\sigma^2
 +O\left(\lambda^3 b_n^3 h^3_n \frac{b^3}{(m+1)^3} \left(\beta_{m+1}^n(b\alpha_1)\right)^3\right)\\
\le & 1+\frac{\lambda^2 b_n^2 \sigma^2}{2} n (m+1)^{-2} (m+1)^{-2b \alpha_1} n^{2b \alpha_1}+O\left(\lambda^3 b_n^3 n^{\frac{3}{2}}(m+1)^{-3} (m+1)^{-3b \alpha_1} n^{3b \alpha_1}  \right)\\
\le& 1+\frac{\lambda^2 b_n^2 \sigma^2}{2} n n^{-2-2b \alpha_1} n^{2b \alpha_1}+O\left(\lambda^3 b_n^3 n^{\frac{3}{2}}n^{-3-3b \alpha_1} n^{3b \alpha_1}  \right)\\
\le& 1+\frac{\lambda^2 b_n^2 \sigma^2}{2n} +O\left(\lambda^3 b_n^3 n^{-\frac{3}{2}}\right).
\end{align*}
yielding by iteration we know that
\begin{align}\label{a-14}
\ee\Big[\exp\left(\lambda b_n  Z_m\right)\Big]&\le \left(1+\frac{\lambda^2 b_n^2 \sigma^2}{2n} +O\left(\lambda^3 b_n^3 n^{-\frac{3}{2}}\right)\right)^{m}\nonumber\\
&\le C\exp\left(\frac{\lambda^2 b_n^2 \sigma^2 m}{2n} \right).
\end{align}
Combining Lemma \ref{lem-1}, (\ref{a-6}), (\ref{a-7}), (\ref{h-lem}), (\ref{a-13}), and (\ref{a-14}), we can get that
\begin{align}\label{a-15}
&\ee\Big[\exp\left(\lambda b_nW_m \right)\Big] \ee\Big[\exp\left(\lambda b_n Z_m \right)\Big]\nonumber\\
=&  C\exp\left(\frac{\lambda^2 b_n^2 \sigma^2 m}{2n} \right) \exp\left(\frac{\lambda^2 b_n^2 h_n^2  \sigma^2}{2}\sum_{j=m+1}^n \frac{b^2}{(j+1)^2} \left(\beta_{j+1}^n(b\alpha_1)\right)^2 \right) \nonumber\\
=& C \exp\left(\frac{\lambda^2 b_n^2 \sigma^2}{2} \left(\frac{m}{n}+h_n^2\sum_{j=m+1}^n \frac{b^2}{(j+1)^2} \left(\beta_{j+1}^n(b\alpha_1)\right)^2\right)\right)\nonumber\\
=&C \exp\left(\frac{\lambda^2 b_n^2 \sigma^2}{2} \left(\frac{m}{n}+1-\frac{\sum_{j=0}^m \frac{1}{(j+1)^2} \left(\beta_{j+1}^n(b\alpha_1)\right)^2}{\sum_{k=0}^n \frac{1}{(k+1)^2} \left(\beta_{k+1}^n(b\alpha_1)\right)^2}\right)\right)\nonumber\\
\le& C \exp\left(\frac{\lambda^2 b_n^2 \sigma^2}{2}\right).
\end{align}
Combining (\ref{a-8}), (\ref{h-lem}), (\ref{a-12}), (\ref{a-15}) and Lemma \ref{lem-1} ,  and using the fact that $b\alpha_1<-1$, we obtain that
\begin{align}\label{a-16}
\sum_{m=0}^n Q_m&\le C \exp\left(\frac{\lambda^2 b_n^2 \sigma^2}{2} \right) \sum_{m=0}^n C\left(\lambda^3 b_n^3 h^3_n \frac{b^3}{(m+1)^3} \left(\beta_{m+1}^n(b\alpha_1)\right)^3\right) \nonumber\\
&\le C \exp\left(\frac{\lambda^2 b_n^2 \sigma^2}{2}\right)\sum_{m=0}^n \lambda^3 b_n^3 n^{\frac{3}{2}}(m+1)^{-3} (m+1)^{-3b \alpha_1} n^{3b \alpha_1}  \nonumber\\
&\le C\lambda^3 b_n^3 n^{-\frac{1}{2}} \exp\left(\frac{\lambda^2 b_n^2 \sigma^2}{2}\right) .
\end{align}
Then from (\ref{a-4}), (\ref{b-1}) and (\ref{a-16}) we know that
\begin{align}
\frac{\ee [\exp\left(\lambda  b_n  \sum_{k=0}^n T_{k+1}\right)]}{\ee [\exp\left(\lambda b_n \sum_{k=0}^n A_{k+1}\right)]}&\le 1+\frac{ C\lambda^3 b_n^3 n^{-\frac{1}{2}}\exp(\frac{\lambda^2 b_n^2 \sigma^2}{2} ) }{\exp\left(\frac{\lambda^2 b_n^2 \sigma^2}{2} \right)}\nonumber\\
&\le \exp\left(C(\lambda^3 b_n^3 n^{-\frac{1}{2}})\right).
\end{align}
Note that $\frac{b_n}{\sqrt{n}}\rightarrow 0$ as $n\rightarrow\infty$, it is not difficult to see that
\begin{align}\label{t-4}
 \lim_{n\rightarrow\infty} \frac{1}{b^2_n} \log \frac{\ee [\exp\left(\lambda  b_n  \sum_{k=0}^n T_{k+1}\right)]}{\ee [\exp\left(\lambda b_n \sum_{k=0}^n A_{k+1}\right)]}&\le \lim_{n\rightarrow\infty}  \frac{1}{b^2_n} C b_n^3 n^{-\frac{1}{2}}\nonumber\\
&=\lim_{n\rightarrow\infty} C\frac{ b_n}{\sqrt{n}}=0.
\end{align}

Next, we will prove that
$$
\lim_{n\rightarrow\infty} \frac{1}{b^2_n} \log \frac{\ee [\exp\left(\lambda b_n \sum_{k=0}^n A_{k+1}\right)]}{\ee [\exp\left(\lambda b_n  \sum_{k=0}^n T_{k+1}\right)]}\le 0.
$$
First, we need the inequality
\beq\label{e}
e^{x}\left(1+\frac{x}{n}\right)^{-x}\le\left(1+\frac{x}{n}\right)^n\le e^x
\deq
hold for all real numbers $x$.
We will focus on proving the first inequality, that is, we need to show that
$$
e^x\le \left(1+\frac{x}{n}\right)^{x+n}.
$$
By substituting $t=\frac{x}{n}$ (i.e. $x=nt$), the inequality becomes
$$
e^{nt}\le \left(1+t\right)^{nt+n}.
$$
Taking the logarithm of both sides, we know that
$$
t\le \left(1+t\right)\ln (1+t).
$$
It is known that for $t>-1$, $\log (1+t)\ge \frac{t}{1+t}$ holds, so we can conclude that (\ref{e}) holds.
Since $\frac{b_n}{\sqrt{n}}\rightarrow 0$, from Lemma \ref{lem-1} and (\ref{h-lem}), it follows that for any sufficiently small $\varepsilon>0$,
$$
\left|\lambda b_n h_n \frac{b}{k+1} \beta_{k+1}^n(b\alpha_1)\right|\le \varepsilon
$$
holds for every $0\le k\le n$.
By choosing $\delta\in (0,1)$ and a constant $C_0$ such that $0<C_0<1$ (which may vary depending on the context),
from Lemma \ref{lem-1}, (\ref{h-lem}), (\ref{a-10}), (\ref{a-17}) and (\ref{e}),   we can conclude that
\begin{align*}
&\ee \left[\exp\left(\lambda b_n \sum_{k=0}^n T_{k+1}\right)\right]\\
=&\prod_{k=0}^n\left(1+\frac{\lambda^2 b_n^2 h^2_n }{2} \frac{b^2}{(k+1)^2} \left(\beta_{k+1}^n(b\alpha_1)\right)^2\sigma^2
 +O\left(\lambda^3 b_n^3 h^3_n \frac{b^3}{(k+1)^3} \left(\beta_{k+1}^n(b\alpha_1)\right)^3\right)\right)\\
 \ge & C_0 \prod_{k=0}^n\left(1+\left(1-\varepsilon\right)\frac{\lambda^2 b_n^2 h^2_n }{2} \frac{b^2}{(k+1)^2} \left(\beta_{k+1}^n(b\alpha_1)\right)^2\sigma^2\right)\\
=& C_0\prod_{k=0}^n\left(1+\left(1-\varepsilon\right)\frac{\lambda^2 b_n^2 \sigma^2 }{2}\frac{(k+1)^{-2} \left(\beta_{k+1}^n(b\alpha_1)\right)^2}{\sum_{m=0}^{n} \left(\frac{\beta_{m+1}^n (b\alpha_1)}{m+1}\right)^2}
 \right)\\
 \ge  &C_0\prod_{k=[\delta n]}^n\left(1+(1-\varepsilon)\frac{\lambda^2 b_n^2 \sigma^2 }{2} \frac{(k+1)^{-2} \left(\beta_{k+1}^n(b\alpha_1)\right)^2}{\sum_{m=0}^{n} \left(\frac{\beta_{m+1}^n (b\alpha_1)}{m+1}\right)^2}
 \right)\\
 \ge &C_0 \prod_{k=[\delta n]}^n\left(1+C_0\frac{\lambda^2 b_n^2 \sigma^2 }{2}  n(\delta n)^{-2-2b\alpha_1} n^{2b\alpha_1}\right)\\
 \ge &C_0\left(1+\frac{\lambda^2 b_n^2 \sigma^2 }{2} \frac{C_0}{n} \right)^{n-\delta n}=\left(1+\frac{\lambda^2 b_n^2 \sigma^2 }{2} \frac{C_0}{n} \right)^{  \frac{n}{C_0}C_0(1-\delta)}\\
 \ge &C_0 \left(\exp\left(\frac{\lambda^2 b_n^2 \sigma^2 }{2}\right)\left(1+\frac{\lambda^2 b_n^2 \sigma^2 }{2} \frac{C_0}{n} \right)^{-\frac{\lambda^2 b_n^2 \sigma^2 }{2}}\right)^{(1-\delta)C_0}\ge C_0 \exp\left(\frac{\lambda^2 b_n^2 \sigma^2 }{2}\right).
\end{align*}
Therefore,  from (\ref{a-4}) we obtain that
\begin{align}\label{t-5}
\lim_{n\rightarrow\infty} \frac{1}{b^2_n} \log \frac{\ee [\exp\left(\lambda b_n \sum_{k=0}^n A_{k+1}\right)]}{\ee [\exp\left(\lambda b_n  \sum_{k=0}^n T_{k+1}\right)]}&\le \lim_{n\rightarrow\infty} \frac{1}{b^2_n} \log \frac{\exp\left(\frac{\lambda^2 b_n^2 \sigma^2}{2}\right)}{C_0 \exp\left(\frac{\lambda^2 b_n^2 \sigma^2}{2}\right)}\nonumber\\
&=\lim_{n\rightarrow\infty} \frac{C}{b_n^2} =0.
\end{align}
Combining (\ref{t-4}) and (\ref{t-5}) yields (\ref{a-5}). Consequently, it follows that (\ref{t-2}) holds, completing the proof.
\end{proof}

\begin{proof}[{\bf Proof of Theorem \ref{thm-sa1}}]
Note that $g(x^*)=0$. From (\ref{sa-0}), we have
\begin{align}\label{p-1-1}
X_{n+1}-x^*&=X_n-x^*+\frac{b}{n+1} \left(g(X_n)+U_{n+1}\right)\nonumber\\
&=X_n-x^*+\frac{b}{n+1}\left[(X_n-x^*)g'(x^*)+\frac{1}{2}g^{''}(\eta_n)(X_n-x^*)^2+U_{n+1}\right]\nonumber\\
&=(X_n-x^*)\left[1+\frac{b}{n+1} g'(x^*)\right]+\frac{b}{2(n+1)}g^{''}(\eta_n)(X_n-x^*)^2+\frac{b}{n+1}U_{n+1},
\end{align}
where $\eta_n=x^*+\theta(X_n-x^*), \theta\in(0,1)$. Let $\alpha_1:=g'(x^*)<0$ and $\beta_{k}^n(b \alpha_1):=\prod_{j=k}^n \left(1+\frac{b \alpha_1}{j+1}\right)$. Iteration of (\ref{p-1-1}) yields
\begin{align}\label{p-1-2}
X_{n+1}-x^*&=\prod_{k=0}^n \left(1+\frac{b \alpha_1}{k+1}\right)\left(X_0-x^*\right)+b\sum_{k=0}^n \frac{1}{k+1}\prod_{i=k+1}^n \left(1+\frac{b \alpha_1}{i+1}\right) U_{k+1}\nonumber\\
&\ \ \ \ +\frac{b}{2}\sum_{k=0}^n \frac{1}{k+1}\prod_{i=k+1}^n \left(1+\frac{b \alpha_1}{i+1}\right)g^{''}(\eta_k)\left(X_k-x^*\right)^2\nonumber\\
&=\beta_{0}^n(b\alpha_1)(X_0-x^*)+\frac{b}{2}\sum_{k=0}^n \frac{1}{k+1}\beta_{k+1}^n(b\alpha_1)g^{''}(\eta_k)(X_k-x^*)^2\nonumber\\
&\ \ \ \ +b\sum_{k=0}^n \frac{1}{k+1}\beta_{k+1}^n(b\alpha_1)U_{k+1}\nonumber\\
&:= I_{1,n}+I_{2,n}+I_{3,n}.
\end{align}
Let
\beq\label{h}
h_n:=\left(b^2\sum_{k=0}^n (k+1)^{-2} \left(\beta_{k+1}^n(b \alpha_1)\right)^2\right)^{-\frac{1}{2}}.
\deq
From Remark \ref{rem-1} we know that $h_n=O\left(\sqrt{n}\right), \ n\rightarrow\infty$.

Note that $b\alpha_1<-1$,  from Lemma \ref{lem-1}, we can get that
$$
h_n|I_{1,n}|=h_n \beta_{0}^n(b\alpha_1)|X_0-x^*|\le C \sqrt{n} n^{b\alpha_1}\rightarrow0.
$$
Then it follows that, for any $r>0$,
\beq\label{b-2}
\lim_{n\rightarrow\infty}\frac{1}{b^2_n} \log \pp (h_n|I_{1,n}|>r b_n)=-\infty.
\deq
Now we take
$$
p_{kn}:=d_n (k+1)^{-b \alpha_1-2} n^{\frac{1}{2}+b \alpha_1}b_n^{1+\gamma}, \ \ \  \gamma>0, \ \ \ 0\le k\le n,
$$
where $\{d_n\}$ is a bounded real sequence, to guarantee that $\sum_{k=0}^n p_{kn}=1$.
From (\ref{sa-1}), for any $r>0$,
\begin{align*}
\pp (h_n|I_{2,n}|>r b_n)&=\pp\left(h_n \frac{b}{2}\sum_{k=0}^n \frac{1}{k+1}\beta_{k+1}^n(b\alpha_1)g^{''}(\eta_k)(X_k-x^*)^2>r b_n\right)\\
&\le \pp\left(h_n \frac{b}{2}\sum_{k=0}^n \frac{1}{k+1}\beta_{k+1}^n(b\alpha_1)K_a(X_k-x^*)^2>r b_n\right)\\
&\le \sum_{k=0}^n \pp\left(h_n \frac{b}{2} \frac{1}{k+1}\beta_{k+1}^n(b\alpha_1)K_a(X_k-x^*)^2>r b_n p_{kn}\right)\\
&=\sum_{k=0}^n \pp\left(\left(X_k-x^*\right)^2>\frac{C_1(k+1) b_n p_{kn}}{h_n \beta_{k+1}^n(b\alpha_1)}\right),
\end{align*}
where $C_1:= \frac{2r}{b K_a}$. On the one hand, by substituting the definition of $p_{kn}$, from Lemma \ref{lem-1} and considering that $b_n=O(n^{\frac{1}{2(1+\gamma)}})$ for some $\gamma>0$, it is easy to check that, for any fixed $k$,
$$
\lim_{n\rightarrow\infty}  \frac{C_1(k+1) b_n p_{kn}}{h_n \beta_{k+1}^n(b\alpha_1)}=\lim_{n\rightarrow\infty}  \frac{b_n^{2+\gamma}}{k+1}=\infty.
$$
On the other hand, from Theorem \ref{lem-5}, there exists a positive constant $N_1$ large enough such that
\begin{align*}
\sum_{k=N_1}^n \pp\left(\left(X_k-x^*\right)^2>\frac{C_1(k+1) b_n p_{kn}}{h_n \beta_{k+1}^n(b\alpha_1)}\right)&\le C \sum_{k=N_1}^n \exp\left(-C k\frac{C_1(k+1) b_n p_{kn}}{h_n \beta_{k+1}^n(b\alpha_1)}\right) \\
&\le C n \exp\left(-C b_n^{2+\gamma}\right).
\end{align*}
In summary, we can conclude that for any $r>0$,
\beq\label{b-3}
\lim_{n\rightarrow\infty}\frac{1}{b^2_n} \log \pp (h_n|I_{2,n}|>r b_n)=-\infty.
\deq
Then it is necessary to prove that, for any $r > 0$,
\begin{align}\label{t-1}
&\lim_{n\rightarrow\infty} \frac{1}{b^2_n} \log \pp (h_n|I_{3,n}|>r b_n)\nonumber\\
=&\lim_{n\rightarrow\infty} \frac{1}{b^2_n} \log \pp \Big(h_n\Big|\sum_{k=0}^n \frac{b}{k+1}\beta_{k+1}^n(b\alpha_1)U_{k+1}\Big|>r b_n\Big)=-\frac{r^2}{2 \sigma^2}.
\end{align}
From Theorem \ref{lem-6} we know that (\ref{t-1}) holds. By integrating  (\ref{p-1-2}), (\ref{b-2}), (\ref{b-3}) and (\ref{t-1}),  we can complete the proof.
\end{proof}

\end{document}